\documentclass[]{article}
\usepackage{graphicx}
\usepackage{mathtools}
\usepackage{amssymb}
\usepackage{amsthm}
\usepackage{hyperref}

\newtheorem{theorem}{Theorem}
\newtheorem{lemma}[theorem]{Lemma}

\begin{document}
\title{Space-Time Finite Element Methods for
Parabolic Evolution Problems
with Non-smooth Solutions\thanks{Supported by the Austrian Science Fund (FWF) under the grant W1214, project DK4.}}
\author{Ulrich Langer\thanks{Institute for Computational Mathematics, Johannes Kepler University Linz.} 
\and
Andreas Schafelner\thanks{Doctoral Program “Computational Mathematics”, Johannes Kepler University Linz.} 
}
\maketitle             
\begin{abstract}
We propose consistent  locally stabilized, conforming finite element schemes on  completely unstructured simplicial space-time meshes 
for the numerical solution of non-autonomous parabolic evolution problems 
under the assumption of maximal parabolic regularity.
We present new a priori estimates for low-regularity solutions.
In order to avoid reduced convergence rates appearing in the case of uniform mesh refinement,
we also consider adaptive refinement procedures based on residual
a posteriori error indicators. 
The huge system of space-time finite element equations is then solved by means of
GMRES preconditioned by algebraic multigrid.\\[1em]
\end{abstract}
%
%

\section{Introduction}\label{sec:intro}
Parabolic initial-boundary value problems of the form 
\begin{equation} 
\label{LS:eq:modelproblem}
\partial_{t}u - \mathrm{div}_x(\nu\, \nabla_{x} u ) = f \;\mbox{in}\; Q, 
\quad
u = 0  \;\mbox{on}\; \Sigma, 
\quad
u = u_0  \;\mbox{on}\; \Sigma_0 
\end{equation}
describe not only heat conduction and diffusion processes but also 
2D eddy current problems in electromagnetics and many other 
evolution processes, where $Q = \Omega \times (0,T)$, $\Sigma = \partial \Omega \times (0,T)$,
and $\Sigma_0 = \Omega\times\{0\}$ denote the space-time cylinder, its lateral boundary,
and the bottom face, respectively.
The spatial computational domain $ \Omega \subset \mathbb{R}^d $, $ d = 1,2,3 $, 
is supposed to be bounded and Lipschitz. The final time is denoted by  $T$.
The right-hand side $f$ is a given source function 
from $L_2(Q)$. The given coefficient $\nu$ may depend on 
the spatial variable $x$ as well as the time variable $t$.
In the latter case, the problem is called non-autonomous.
We suppose at least that $\nu$ is uniformly positive and bounded almost everywhere.
We here consider homogeneous Dirichlet boundary conditions for the sake of simplicity.
In practice, we often meet mixed boundary conditions. 
Discontinuous coefficients, non-smooth boundaries, changing boundary conditions,
non-smooth or incompatible initial conditions, and non-smooth right-hand sides 
can lead to non-smooth solutions.

In contrast to the conventional time-stepping methods in combination with 
some spatial discretization method, or the more advanced, but closely related discontinuous Galerkin (dG) 
methods based on time slices, 
we here consider space-time finite element discretizations treating time as just another variable 
and the term $\partial_t u$  in (\ref{LS:eq:modelproblem}) as convection term in time.
Following \cite{LS:LangerNeumuellerSchafelner:2019a}, we derive consistent, locally stabilized, conforming finite element schemes on  completely unstructured simplicial space-time meshes 
under the assumption of maximal parabolic regularity; see, e.g., \cite{LS:Fackler:2017a}.
Unstructured space-time schemes have clear advantages with respect to adaptivity, 
parallelization, and the numerical treatment of moving interfaces or special domains.
We refer the reader to the survey paper
\cite{LS:SteinbachYang:2018a}
that provides an excellent overview of  completely unstructured space-time methods and simultaneous space-time adaptivity.
In particular, we would like to mention the papers 
\cite{LS:Steinbach:2015a} that is based on 
an 
inf-sup-condition,
\cite{LS:DevaudSchwab:2018a} that uses mesh-grading in time,
and 
\cite{LS:BankVassilevskiZikatanov:2016a} 
that also uses stabilization techniques. All three papers treat the autonomous case.

We here present new a priori 
discretization error
estimates for low-regularity solutions.
In order to avoid reduced convergence rates appearing in the case of uniform mesh refinement,
we also consider adaptive refinement procedures 
in the numerical experiments presented in Section~\ref{sec:num}.
The adaptive refinement procedures are based on residual
a posteriori error indicators. 
The huge system of space-time finite element equations is then solved by means of
Generalized Minimal Residual Method (GMRES) preconditioned by an algebraic multigrid cycle. 
In particular, in the 4D space-time 
case that is 3D in space, 
simultaneous space-time adaptivity and  parallelization can considerably reduce
the computational time.
The space-time finite element solver was implemented 
in the framework of MFEM. 
The numerical results nicely confirm our theoretical findings.
The parallel version of the code shows an excellent 
parallel performance.

\section{Weak formulation and maximal parabolic regularity}
%
The weak formulation of the model problem \eqref{LS:eq:modelproblem} reads as follows: find $ u\in H^{1,0}_0(Q):=\{u\in L_2(Q):\nabla_x u\in [L_2(Q)]^d,\,  u=0\ \text{on}\ \Sigma \} $ such that (s.t.)
\begin{equation}\label{LS:eq:weakformulation}
\int_{Q} \bigl( -u\,\partial_{t} v +\nu\, \nabla_{x}u\cdot\nabla_{x}v\bigr)\;\mathrm{d}(x,t) 
= \int_{Q}\!f\,v\;\mathrm{d}(x,t) + \int_{\Omega}\!u_0\,v|_{t=0}\;\mathrm{d}x 
\end{equation}
for all 
$ v \in \hat{H}^{1}_0(Q) = \{ u\in H^1(Q) : u=0\ \text{on}\ \Sigma\cup\Sigma_T \} $, where $ \Sigma_T:=\Omega\times\{T\} $. The existence and uniqueness of weak solutions is well understood; 
see, e.g., \cite{LS:Ladyzhenskaya:1985a}.
It was already shown in \cite{LS:Ladyzhenskaya:1985a} that $\partial_t u \in L^2(Q)$ and 
$\Delta u \in L^2(Q)$ provided that $\nu =1$, $f \in L^2(Q)$, and $u_0 = 0$.
This case is called maximal parabolic regularity. Similar results can be 
obtained under more general assumptions imposed on the data; see, e.g., \cite{LS:Fackler:2017a}
for some very recent results on the non-autonomous case.
%
\section{Locally stabilized space-time finite element\\ methods}
%
In order to derive the space-time finite element scheme, we need an admissible, shape regular 
decomposition $ \mathcal{T}_h = \{K\}$ of the space-time cylinder $ Q = \bigcup_{K\in\mathcal{T}_h}\overline{K}$ 
into finite elements $K$. On $ \mathcal{T}_h$, we define a $H^1$ conforming finite element space 
$V_{h}$ 
by means of polynomial simplicial finite elements of the degree $p$
in the usual way; see, e.g., \cite{LS:BrennerScott:2008a}.
Let us assume that the solution $u$ of (\ref{LS:eq:weakformulation}) belongs to 
the space $V_0 = H^{\mathcal{L},1}_{0,\underline{0}}(\mathcal{T}_h) := \{ u\in L_2(Q) : \partial_t u\in L_2(K),\ \mathcal{L}u:=\mathrm{div}_x(\nu\nabla_{x}u)\in L_2(K)\ \forall K\in\mathcal{T}_h,\ \text{and}\ u|_{\Sigma\cup\Sigma_0}=0 \}$, i.e., we only need some local version of maximal parabolic regularity,
and, for simplicity, we assume homogeneous initial conditions, i.e., $u_0 = 0$.
Multiplying the PDE (\ref{LS:eq:modelproblem}) on $K$ by a local time-upwind test function 
$ v_h + \theta_K h_K\partial_{t}v_h $, 
with $ v_h\in V_{0h} =  \{v_h \in V_{h}: v_h = 0 \, \mbox{on} \, \Sigma \cup \Sigma_0\}$, $ h_K = \mathrm{diam}(K) $, and a parameter $ \theta_K > 0 $  
which we will specify later, integrating over $K$, integrating by parts,
and summing up over all elements  $K\in\mathcal{T}_h$, we arrive at the following consistent 
space-time finite element scheme: find $u_h \in V_{0h}$ s.t.
\begin{equation}\label{eq:fe-scheme}
a_h(u_h,v_h) = l_h(v_h),\quad \forall v_h\in V_{0h},
\end{equation}
with
\begin{align}
a_h(u_h,v_h) =& \sum_{K\in\mathcal{T}_h} \int_{K} \big[ \partial_t u_h v_h + \theta_K h_K \partial_t u_h \partial_t v_h \label{LS:eq:bilinearform} \\
&\qquad\qquad + \nu \nabla_{x}u_h\cdot\nabla_{x}v_h - \theta_K h_K \mathrm{div}_x(\nu \nabla_{x} u_h)\partial_t v_h \big] \mathrm{d}(x,t),\notag\\
l_h(v_h) =& \sum_{K\in\mathcal{T}_h} \int_{K} f v_h + \theta_K h_K f \partial_t v_h\mathrm{d}(x,t). \label{LS:eq:linearform}
\end{align}
The bilinearform $ a_h(\,.\,,\,.\,) $ is coercive on $ V_{0h}\times V_{0h} $ wrt to the norm
\begin{equation}
\| v \|_{h}^2 = \frac{1}{2}\| v \|_{L_2(\Sigma_T)}^2 + \sum_{K\in\mathcal{T}_h} \Bigl[ \theta_{K} h_K\| \partial_{t} v \|_{L_2(K)}^2 + \| \nabla_x v \|_{L_2^{\nu}(K)}^2 \Bigr], \label{LNS:eq:norm:h}
\end{equation}
i.e.,
$ a_h(v_h,v_h) \ge \mu_c \|v_h\|_h^2$, 
$\forall v_h\in V_{0h}$,
and bounded on $ V_{0h,*}\times V_{0h} $ wrt to the norm
\begin{align}
\| v \|_{h,*}^2 &=
\|v\|_{h}^2 + \sum_{K\in\mathcal{T}_h} \Bigl[  (\theta_{K} h_K)^{-1}\| v \|_{L_2(K)}^2 + \theta_{K} h_K \|\mathrm{div}_x (\nu \nabla_{x} v) \|_{L_2(K)}^2 \Bigr],\label{LNS:eq:norm:h,*}
\end{align}
i.e.,
$a_h(u_h,v_h) \le \mu_b \|u_h\|_{h,*}\|v_h\|_{h}$, $\forall u_h\in V_{0h,*},\forall v_h\in V_{0h}$,
where $ V_{0h,*}:= H^{\mathcal{L},1}_{0,\underline{0}}(Q) + V_{0h} $; see \cite[Lemma 3.8]{LS:LangerNeumuellerSchafelner:2019a} and \cite[Remark 3.13]{LS:LangerNeumuellerSchafelner:2019a}, respectively. The coercivity constant $ \mu_c $ is robust in $ h_K $ 
provided that we choose $ \theta_{K} = \mathcal{O}(h_K) $; see Section~\ref{sec:num} or \cite[Lemma 3.8]{LS:LangerNeumuellerSchafelner:2019a} for the explicit choice.
From the above derivation of the scheme, we  get consistency
$a_h(u,v_h) = l_h(v_h), \; \forall v_h\in V_{0h}$,
provided that the solution $u$ belongs to  $H^{\mathcal{L},1}_{0,\underline{0}}(Q) $
that is ensured in the case of maximal parabolic regularity.
The space-time finite element scheme (\ref{eq:fe-scheme}) 
and the consistency relation 
immediately yield Galerkin orthogonality 
\begin{equation}\label{eq:GalerkinOrthogonality}
a_h(u - u_h,v_h) = 0, \quad \forall v_h\in V_{0h}.
\end{equation}
We deduce that \eqref{eq:fe-scheme} is nothing but a huge linear system of algebraic equations. Indeed, let $ V_{0h} = \mathrm{span}\{ p^{(j)}, j=1,\dots,N_h \} $, where $ \{ p^{(j)}, j=1,\dots,N_h \} $ is the nodal finite element basis and $ N_h $ is the total number of space-time degrees of freedom (dofs). Then we can express each function in $ V_{0h} $ in terms of this basis, i.e., we can identify each finite element function $ v_h\in V_{0h} $ with its coefficient vector $ \mathbf{v}_h\in \mathbb{R}^{N_h} $. Moreover, each basis function $ p^{(j)} $ is also a valid test function. Hence, we obtain $ N_h $ equations from \eqref{eq:fe-scheme}, which we rewrite as a system of linear algebraic equations, i.e.
\[ \mathbf{K}_h\,\mathbf{u}_h = \mathbf{f}_h, \]
with the solution vector
$\mathbf{u}_{h}= (u_j)_{j=1,\ldots,N_h}  $,
the vector
$ \mathbf{f}_{h} = \bigl(l_h(p^{(i)})\bigr)_{i=1,\ldots,N_h}$,
and system matrix
$ \mathbf{K}_{h} = \bigl(a_h(p^{(j)},p^{(i)})\bigr)_{i,j=1,\ldots,N_h} $
that is non-symmetric, but positive definite.
%
\section{A priori discretization error estimates}
%
Galerkin orthogonality (\ref{eq:GalerkinOrthogonality}),
together with coercivity and boundedness, enables us to prove a C\`ea-like estimate, where we bound the discretization error in the $ \|\,.\,\|_h $-norm by the best-approximation error in the $ \|\,.\,\|_{h,*} $-norm.
\begin{lemma}\label{LS:lem:cea-like}
    Let the bilinearform \eqref{LS:eq:bilinearform} be coercive \cite[Lemma 3.8]{LS:LangerNeumuellerSchafelner:2019a} with constant $ \mu_c $ and bounded \cite[Lemma 3.11, Remark 3.13]{LS:LangerNeumuellerSchafelner:2019a} with constant $ \mu_b $, and let $ u\in H^{\mathcal{L},1}_{0,\underline{0}}(\mathcal{T}_h) $ be the solution of the space-time variational problem \eqref{LS:eq:weakformulation}. Then there holds
    \begin{equation}\label{LS:eq:cea-like}
    \|u-u_h\|_h\leq \biggl(1+\frac{\mu_b}{\mu_c}\biggr) \inf_{v_h\in V_{0h}} \|u-v_h\|_{h,*},
    \end{equation}
    where $ u_h\in V_{0h} $ is the solution to the space-time finite element scheme \eqref{eq:fe-scheme}.
\end{lemma}
\begin{proof}
    Estimate (\ref{LS:eq:cea-like}) easily follows from triangle inequality and Galerkin-ortho\-go\-na\-lity; see \cite[Lemma 3.15, Remark 3.16]{LS:LangerNeumuellerSchafelner:2019a} for details.
\end{proof}
Next, we estimate the best approximation error by the interpolation error, where we have to choose a proper interpolation operator $ \mathfrak{I}_* $. 
For smooth solutions, i.e., $ u\in H^l(Q) $ with $ l > (d+1)/2 $, we obtained a localized a priori error estimate, see \cite[Theorem 3.17]{LS:LangerNeumuellerSchafelner:2019a}, where we used the standard Lagrange interpolation operator $ \mathcal{I}_h $; see e.g. \cite{LS:BrennerScott:2008a}. 
In this paper, we are interested in non-smooth solutions, which means that we only require $ u\in H^{l}(Q) $, with some real $ l > 1 $. Hence, we cannot use the Lagrange interpolator. We can, however, use so-called quasi-interpolators, e.g. Cl\'{e}ment \cite{LS:Clement:1975a} or Scott-Zhang \cite{LS:BrennerScott:2008a}. For this kind of operators, we need a neighborhood $ S_K $ of an element $ K\in\mathcal{T}_h $ which is defined as 
$ S_K := \{ K'\in\mathcal{T}_h : \overline{K}\cap\overline{K}'\neq\emptyset\}. $
Let $ u\in H^{l}(Q) $, with some real $ l>1 $, 
then, for the Scott-Zhang quasi-interpolation operator $ \mathfrak{I}_{S\!Z} : L_2(Q) \rightarrow V_{0h} $, we have the local estimate
(see e.g. \cite[(4.8.10)]{LS:BrennerScott:2008a})
\begin{equation}\label{LS:eq:quasi-interpolation}
\| v-\mathfrak{I}_{S\!Z} v \|_{H^k(K)} \leq C_{\mathfrak{I}_{S\!Z}} h_K^{l-k} |v|_{H^l(S_K)},\ k=0,1.
\end{equation}
For details on
the construction of such a quasi-interpolator, we refer to \cite{LS:BrennerScott:2008a} and the references therein. For simplicity, we now assume that the diffusion coefficient $ \nu $ is piecewise constant, i.e., $ \nu|_K = \nu_K $, for all $ K\in\mathcal{T}_h $. Then we can show the 
following lemma.
\begin{lemma} 
    \label{LS:lem:Quasi-InterpolationerrorEstimates}
    Let $ l > 1 $ and $ v\in V_0\cap H^l(\mathcal T_h) $. Then the following interpolation error estimates are valid:
    \begin{align}
    \| v-{\mathfrak{I}_{S\!Z}} v \|_{L_2(\Sigma_T)} &\leq c_1\Bigl(\sum_{\mathclap{\substack{
                K\in\mathcal{T}_h\\ \partial K\cap\Sigma_T\neq\emptyset}}}h_K^{2s-1}| v |_{H^{s}(K)}^2\Bigr)^{1/2}, \label{LNS:eq:interpolation:boundary}\\
    \| v-{\mathfrak{I}_{S\!Z}} v \|_h &\leq c_2 \Bigl(\sum_{K\in\mathcal{T}_h}h_K^{2(s-1)} |v|_{H^s(S_K)}^2\Bigr)^{1/2}, \label{LNS:eq:interpolation:h} \\
    \| v-{\mathfrak{I}_{S\!Z}} v \|_{h,*} &\leq c_3 \Bigl(\sum_{K\in\mathcal{T}_h}h_K^{2(s-1)} \bigl(|v|_{H^s(S_K)}^2 + \|\mathrm{div}_x(\nu\nabla_{x}v)\|_{L_2(K)}^2\bigr)\Bigr)^{1/2}, \label{LNS:eq:interpolation:h,*}
    \end{align}
    with $ s = \min\{l,p+1\} $ and positive constants $ c_1, c_2$ and $c_3 $, that do not depend on $ v $  or $ h_K $
    provided that $ \theta_K=\mathcal{O}(h_K) $ for all $K \in \mathcal{T}_h$. Here, $ p $ denotes the polynomial degree of the finite element shape functions 
    on the reference element.
\end{lemma}
\begin{proof}
    For the first estimate, we use the scaled trace inequality and the quasi-interpolation estimate \eqref{LS:eq:quasi-interpolation} with $ k = 0,1 $
    \begin{align*}
    \| v-\mathfrak{I}_{S\!Z} v \|_{L_2(\Sigma_T)}^2 &= \sum_{\mathclap{\substack{K\in\mathcal{T}_h\\ \partial K\cap\Sigma_T\neq\emptyset}}} \| v-\mathfrak{I}_{S\!Z} v \|_{L_2(\partial K\cap\Sigma_T)}^2 \leq \sum_{\mathclap{\substack{K\in\mathcal{T}_h\\ \partial K\cap\Sigma_T\neq\emptyset}}} \| v-\mathfrak{I}_{S\!Z} v \|_{L_2(\partial K)}^2 \\
    &\leq \sum_{\mathclap{\substack{K\in\mathcal{T}_h\\ \partial K\cap\Sigma_T\neq\emptyset}}} \bigl[ c_{Tr}^2h_K^{-1}(\| v-\mathfrak{I}_{S\!Z} v \|_{L_2(K)}^2 + h_K^2 \| \nabla (v-\mathfrak{I}_{S\!Z} v) \|_{L_2(K)}^2)  \bigr] \\
    &\leq c_{Tr}^2 \sum_{\mathclap{\substack{K\in\mathcal{T}_h\\ \partial K\cap\Sigma_T\neq\emptyset}}} \bigl[C_{\mathfrak{I}_{S\!Z}}^2\ h_K^{-1}h_K^{2l} | v |_{H^{l}(S_K)}^2 +C_{\mathfrak{I}_{S\!Z}}^2\  h_K\,h_K^{2(l-1)} | v |_{H^{l}(S_K)}^2 \bigr]\\
    & \leq \max_{K\in\mathcal{T}_h}\bigl(2\,c_{Tr}^2 C_{\mathfrak I_{S\!Z}}^2\bigr)\sum_{\mathclap{\substack{K\in\mathcal{T}_h\\ \partial K\cap\Sigma_T\neq\emptyset}}} \bigl[h_K^{2l-1} | v |_{H^{l}(S_K)}^2 \bigr],
    \end{align*}
    which corresponds to \eqref{LNS:eq:interpolation:boundary} with $ c_1 = \max_{K\in\mathcal{T}_h}\bigl(2\,c_{Tr}^2 C_{\mathfrak I_{S\!Z}}^2\bigr) $.
    To show the second estimate \eqref{LNS:eq:interpolation:h}, we use definition \eqref{LNS:eq:norm:h} and that $ \nu $ is piecewise constant, the quasi-interpolation error estimate \eqref{LS:eq:quasi-interpolation} with $ k = 1 $,  and the above estimate \eqref{LNS:eq:interpolation:boundary}, and obtain 
    \begin{align*}
    \| v - &\mathfrak{I}_{S\!Z} v \|_h^2 \\
    &= \sum_{K\in\mathcal{T}_h} \Bigl[\theta_{K}h_K\,\|\partial_t (v-\mathfrak{I}_{S\!Z} v) \|_{L_2(K)}^2 + \|\nu^{1/2}\nabla_{x}(v-\mathfrak{I}_{S\!Z}v) \|_{L_2(K)}^2\Bigr]\\
    &\quad+ \frac{1}{2} \|v-\mathfrak{I}_{S\!Z}v\|_{L_2(\Sigma_T)}^2\\
    &\leq \sum_{K\in\mathcal{T}_h}\Bigl[\theta_K h_K C_{\mathfrak{I}_{S\!Z}}^2 h_K^{2(l-1)} |v|_{H^l(S_K)}^2 + {\nu}_K C_{\mathfrak{I}_{S\!Z}}^2 h_K^{2(l-1)}|v|_{H^l(S_K)}\Bigr] \\
    &\quad+ c_1 \sum_{\mathclap{\substack{K\in\mathcal{T}_h\\ \partial K\cap\Sigma_T\neq\emptyset}}} h_K^{2l-1} | v |_{H^{l}(S_K)}^2\\
    &\leq \sum_{K\in\mathcal{T}_h} \bigl(C_{\mathfrak{I}_{S\!Z}}^2(\theta_K h_K + {\nu}_K) + c_1 h_K \bigr) h_K^{2(l-1)} | v |_{H^{l}(S_K)}^2 \leq c_2 \sum_{K\in\mathcal{T}_h} h_K^{2(l-1)} | v |_{H^{l}(S_K)}^2,
    \end{align*}
    with $ c_2 = \max_{K\in\mathcal{T}_h}\bigl(C_{\mathfrak{I}_{S\!Z}}^2(\theta_K h_K + {\nu}_K) + c_1 h_K \bigr)  $.
    For the third estimate, we 
    deduce from \eqref{LNS:eq:norm:h,*} that we only have to estimate the additional sum
    \[
    \sum_{K\in\mathcal{T}_h}\Bigl[(\theta_K h_K)^{-1} \|v - \mathfrak{I}_{S\!Z} v\|_{L_2(K)} ^2 + \theta_K h_K \|\mathrm{div}_x(\nu\nabla_{x}(v-\mathfrak{I}_{S\!Z} v))\|_{L_2(K)}^2\Bigr]. 
    \]
    We start with the first $ L_2 $-term. We apply the quasi-interpolation estimate \eqref{LS:eq:quasi-interpolation} with $ k = 0 $ and obtain
    \begin{align*}
    \sum_{K\in\mathcal{T}_h}(\theta_K h_K)^{-1} \|v - \mathfrak{I}_{S\!Z} v\|_{L_2(K)} ^2 &\leq \sum_{K\in\mathcal{T}_h} (\theta_K h_K)^{-1} C_{\mathfrak{I}_{S\!Z}}^2 h_K^{2l} |v|_{H^l(S_K)}^2\\
    &\leq \sum_{K\in\mathcal{T}_h}C_{\mathfrak{I}_{S\!Z}}^2 \Bigl(\frac{h_K}{\theta_K}\Bigr) h_K^{2(l-1)} |v|_{H^l(S_K)}^2.
    \end{align*}
    Note that the term $ (h_K/\theta_K) $ is bounded for $ \theta_K = \mathcal{O}(h_K) $. For the $ L_2 $-norm of the spatial divergence, we can distinguish between two cases: linear basis functions ($ p=1 $) and higher order basis functions ($ p\geq2 $). For linear basis functions, we split the divergence of the gradient, obtaining 
    \begin{equation*}
    \|\mathrm{div}_x(\nu\nabla_{x}(v-\mathfrak{I}_{S\!Z} v))\|_{L_2(K)}^2 = 
    \|\mathrm{div}_x(\nu\nabla_{x}v) - \mathrm{div}_x(\nu\nabla_{x}(\mathfrak{I}_{S\!Z} v))\|_{L_2(K)}^2
    \end{equation*}
    for each element $ K\in\mathcal{T}_h $.
    Since $ \mathfrak{I}_{S\!Z}v $ is a linear polynomial and $ \nu $ is piecewise constant, we deduce
    \[ \mathrm{div}_x(\nu\nabla_{x}(\mathfrak{I}_{S\!Z}v)) = \nu_K\,\mathrm{div}_x(\nabla_{x}(\mathfrak{I}_{S\!Z}v)) = 0 \]
    for each element $ K\in\mathcal{T}_h $.
    Hence, we get
    \begin{align*}
    \sum_{K\in\mathcal{T}_h} \theta_K h_K \|\mathrm{div}_x(\nu\nabla_{x}v)\|_{L_2}^2 = \sum_{K\in\mathcal{T}_h} h_K^{2(l-1)} \theta_K h_K^{1-2(l-1)} \|\mathrm{div}_x(\nu\nabla_{x}v)\|_{L_2(K)}^2 ,
    \end{align*}
    where the norm is bounded since $ v \in V_0 $. Moreover, for $ \theta_K = \mathcal{O}(h_K) $, the term $ \theta_K h_K^{1-2(l-1)} $ is bounded for $ 1 \leq l \leq 2 $, and the case $ l > 2 $ is already treated in \cite{LS:LangerNeumuellerSchafelner:2019a}.
    Combining all above estimates, we obtain
    \begin{align*}
    \|v-\mathfrak{I}_{S\!Z} v\|_{h,*}^2 \leq&c_2 \sum_{K\in\mathcal{T}_h} h_K^{2(l-1)} |v|_{H^l(S_K)}^2 + \sum_{K\in\mathcal{T}_h} C_{\mathfrak{I}_{S\!Z}}^2 \Bigl(\frac{h_K}{\theta_K}\Bigr) h_K^{2(l-1)} |v|_{H^l(S_K)}^2 \\
    &+ \sum_{K\in\mathcal{T}_h} (\theta_{K}h_K^{3-2l})  h_K^{2(l-1)}\|\mathrm{div}_x(\nu\nabla_{x}v)\|_{L_2(K)}^2\\
    \leq& c_3 \sum_{K\in\mathcal{T}_h} h_K^{2(l-1)} \Bigl(|v|_{H^l(S_K)}^2 + \|\mathrm{div}_x(\nu\nabla_{x}v)\|_{L_2(K)}^2\Bigr),
    \end{align*}
    with $ c_3 = c_2 + \max_{K\in\mathcal{T}_h}\bigl\{ C_{\mathfrak{I}_{S\!Z}}^2(h_K/\theta_{K}), \theta_{K}h_K^{3-2l} \bigr\} $.
    For the general case of higher order basis functions, i.e., $ p\geq2 $, the divergence of the gradient does not vanish. First we split the divergence of the gradient and also the norms, obtaining
    \begin{align*}
    \sum_{K\in\mathcal{T}_h} \theta_K h_K &\|\mathrm{div}_x(\nu\nabla_{x}v) - \mathrm{div}_x(\nu\nabla_{x}(\mathfrak{I}_{S\!Z} v))\|_{L_2}^2 \\
    &\leq \sum_{K\in\mathcal{T}_h} 2\,\theta_K h_K \left( \|\mathrm{div}_x(\nu\nabla_{x}v)\|_{L_2}^2 +  \| \mathrm{div}_x(\nu\nabla_{x}(\mathfrak{I}_{S\!Z} v))\|_{L_2}^2 \right)
    \end{align*}
    The first term in the sum we have already estimated above, where we now replace $ l $ by $ s = \min\{l,p+1\} $, which is in our case ($ l\leq2 $) again just $ l $. For the second term, we insert $ \nu_K \mathrm{div}_x (\nabla_{x}(\mathfrak{I}_{S\!Z}^1 v)) $ into the norm, where $ \mathfrak{I}_{S\!Z}^1 v $ is the linear quasi-interpolation of $ v $. We observe that $ \mathfrak{I}_{S\!Z}^1 v - \mathfrak{I}_{S\!Z} v $ is a finite element function, i.e., we can apply the inverse equality for the $ H(\mathrm{div}_x) $-norm \cite[Lemma~3.5]{LS:LangerNeumuellerSchafelner:2019a}. Since the diffusion coefficient is piecewise constant, we obtain
    \begin{align*}
    \sum_{K\in\mathcal{T}_h} 2\,\theta_K h_K &\| \mathrm{div}_x(\nu\nabla_{x}(\mathfrak{I}_{S\!Z}^1 v - \mathfrak{I}_{S\!Z} v))\|_{L_2}^2 \\
    &\leq \sum_{K\in\mathcal{T}_h} 2\,\theta_K h_K h_K^{-2} c_{I,3}^2 \|\nu\nabla_{x}(\mathfrak{I}_{S\!Z}^1 v - \mathfrak{I}_{S\!Z} v) \|_{L_2(K)}^2  \\
    &\leq \sum_{K\in\mathcal{T}_h} 2\,\theta_K h_K^{-1} c_{I,3}^2 {\nu}_K^2 \|\nabla_{x}(\mathfrak{I}_{S\!Z}^1 v - \mathfrak{I}_{S\!Z} v) \|_{L_2(K)}^2
    \end{align*}
    Now we insert and subtract $ \nabla_{x} v $ and use the triangle inequality, which yields
    \begin{align*}
    \sum_{K\in\mathcal{T}_h} 2\,&\theta_K h_K^{-1} c_{I,3}^2 \overline{\nu}_K^2 \|\nabla_{x}(\mathfrak{I}_{S\!Z}^1 v - \mathfrak{I}_{S\!Z} v) \|_{L_2(K)}^2 \\
    &\leq \sum_{K\in\mathcal{T}_h} 4\,\theta_K h_K^{-1} c_{I,3}^2 \overline{\nu}_K^2\left( \|\nabla_{x}(v - \mathfrak{I}_{S\!Z}^1 v)\|_{L_2(K)}^2 + \| \nabla_{x}( v-  \mathfrak{I}_{S\!Z} v) \|_{L_2(K)}^2 \right).
    \end{align*}
    For both terms we can apply \eqref{LS:eq:quasi-interpolation} (quasi-interpolation) and obtain
    \begin{align*}
    \sum_{K\in\mathcal{T}_h}& 4\,\theta_K h_K^{-1} c_{I,3}^2{\nu}_K^2 \Bigl(  \|\nabla_{x}(v - \mathfrak{I}_{S\!Z}^1 v)\|_{L_2(K)}^2 + \| \nabla_{x}( v-  \mathfrak{I}_{S\!Z} v) \|_{L_2(K)}^2 \Bigr)\\
    &\leq \sum_{K\in\mathcal{T}_h} 4\,\theta_K h_K^{-1} c_{I,3}^2{\nu}_K^2 \Bigl( C_{\mathfrak{I}_{S\!Z}}^2 h_K^{2(l-1)} |v|_{H^l(S_K)}^2 + C_{\mathfrak{I}_{S\!Z}}^2 h_K^{2(l-1)}| v|_{H^l(S_K)}^2 \Bigr)\\
    &\leq \sum_{K\in\mathcal{T}_h} {8\,\theta_K h_K^{-1} c_{I,3}^2{\nu}_K^2 C_{\mathfrak{I}_{S\!Z}}^2}
    h_K^{2(l-1)} |v|_{H^l(S_K)}^2.
    \end{align*}
    Combining all of the above, 
    \begin{align*}
    \|v-&\mathfrak{I}_{S\!Z} v\|_{h,*}^2 \\
    &\leq c_2 \sum_{K\in\mathcal{T}_h} h_K^{2(l-1)} |v|_{H^l(S_K)} + \sum_{K\in\mathcal{T}_h} 2(\theta_{K}h_K^{3-2l}) h_K^{2(l-1)}\|\mathrm{div}_x(\nu\nabla_{x}v)\|_{L_2(K)}^2  \\
    &\qquad+ \sum_{K\in\mathcal{T}_h} \bigl({8\,\theta_K h_K^{-1} c_{I,3}^2{\nu}_K^2 C_{\mathfrak{I}_{S\!Z}}^2}\bigr) h_K^{2(l-1)}|v|_{H^l(S_K)}^2\\
    &\leq c_{3,p} \sum_{K\in\mathcal{T}_h} h_K^{2(l-1)} \left(|v|_{H^l(S_K)}^2 + \|\mathrm{div}_x(\nu\nabla_{x}v)\|_{L_2(K)}^2\right),
    \end{align*}
    where $ c_{3,p} = c_2 + \max_{K\in\mathcal{T}_h}\bigl\{{8\,\theta_K h_K^{-1} c_{I,3}^2{\nu}_K^2 C_{\mathfrak{I}_{S\!Z}}^2},2(\theta_{K}h_K^{3-2l})  \bigr\}$ .
\end{proof}
Now we are in the position to prove our main theorem.
\begin{theorem}
    Let $ p $ be the polynomial degree used, and let $ u\in H^{l}(Q)\cap V_{0} $, with $ l>1 $, be the exact solution,
    and $ u_h\in V_{0h} $ be the approximate solution of the finite element scheme \eqref{eq:fe-scheme}. Furthermore, let the assumptions of Lemma~\ref{LS:lem:cea-like} (C\'{e}a-like estimate) and \ref{LS:lem:Quasi-InterpolationerrorEstimates} (quasi-interpolation estimates) hold.
    Then the a priori discretization error estimate
    \begin{equation}\label{LS:eq:discretization-error}
    \|u-u_h\|_h \le C\,\Bigl(\sum_{K\in\mathcal{T}_h} h_K^{2(s-1)}\bigl(|u|_{H^s(S_K)} + \|\mathrm{div}_x(\nu\nabla_{x}u)\|_{L_2(K)}^2\bigr) \Bigr)^{1/2},
    \end{equation}
    holds, 
    with $ s=\min\{l,p+1\} $ and a positive generic constant $ C $.
\end{theorem}
\begin{proof}
    Choosing the quasi-interpolant $ v_h = \mathfrak{I}_{S\!Z}u $ 
    in \eqref{LS:eq:cea-like}, we can apply the quasi-interpolation estimate  \eqref{LNS:eq:interpolation:h,*} to obtain
    \begin{align*}
    \|u-u_h\|_h &\leq \Bigl(1+\frac{\mu_b}{\mu_c}\Bigr) \|u-\mathfrak I_{S\!Z}u\|_{h,*}\\
    &\le c_3\Bigl(1+\frac{\mu_b}{\mu_c}\Bigr) \Bigl(\sum_{K\in\mathcal{T}_h}h_K^{2(s-1)} \bigl(|u|_{H^s(S_K)}^2 + \|\mathrm{div}_x(\nu\nabla_{x}u)\|_{L_2(K)}^2\bigr)\Bigr)^{1/2}.
    \end{align*}
    We set $ C = c_3(1+\mu_b/\mu_c) $ to obtain \eqref{LS:eq:discretization-error}, which closes the proof.
\end{proof}
%

%
\section{Numerical results}\label{sec:num}
%
We implemented the space-time finite element scheme \eqref{eq:fe-scheme} in \texttt{C++}, where we used the finite element library MFEM\footnote{\url{http://mfem.org/}} and the solver library \textit{hypre}\footnote{\url{https://www.llnl.gov/casc/hypre/}}. The linear system was solved by means of the GMRES method, preconditioned by the algebraic multigrid \textit{BoomerAMG}. We stopped the iterative procedure if the initial residual was reduced by a factor of $ 10^{-8} $. Both libraries are already fully parallelized with the Message Passing Interface (MPI). Therefore, we performed all numerical tests on the distributed memory cluster RADON1\footnote{\url{https://www.ricam.oeaw.ac.at/hpc/}}
in Linz. For each element $ K\in\mathcal{T}_h $, we choose $ \theta_{K} = h_K/(\tilde{c}^2 \overline{\nu}_K) $, where $ \tilde{c} $ is computed by solving a local generalized eigenvalue problem which comes from an inverse inequality, see \cite{LS:LangerNeumuellerSchafelner:2019a} for further details.
\subsection{Example: Highly oscillatory solution}\label{sec:num:1}
As first test example, 
we consider the unit (hyper-)cube $ Q = (0,1)^{d+1} $, with $ d=2,3 $, as space-time cylinder, 
and
$ \nu\equiv1 $. 
The manufactured function
\[ u(x,t) = \sin\!\left(\frac{1}{\frac{1}{10\pi}+\sqrt{\sum_{i=1}^{d}x_i^2 + t^2}}\right) \]
serves as the exact solution, where we compute the right hand side $ f $ accordingly. This solution is 
highly oscillatory. Hence, we do not expect optimal rates for uniform refinement in the pre-asymptotic range. However, using adaptive refinement, we may be able to recover the optimal rates.
We used the residual based error indicator proposed by Steinbach and Yang in \cite{LS:SteinbachYang:2018a}. For each element $ K\in\mathcal{T}_h $, we compute
\[ \eta_K^2 := h_K^2 \|R_h(u_h)\|_{L_2(K)}^2 + h_K \|J_h(u_h) \|_{L_2(\partial K)}^2, \]
where $ u_h $ is the solution of 
\eqref{eq:fe-scheme}, $ R_h(u_h) := f + \mathrm{div}_x(\nu \nabla_x u_h)-\partial_t u_h $ in $K$, and $ J_h(u_h) := [\nu \nabla_x u_h]_e$ on $ e\subset\partial K$, with $ [\,.\,]_e $ denoting the jump across one face $ e\subset \partial K $. 
We mark each element where the condition $ \eta_K \geq \sigma \max_{K\in\mathcal{T}_h}  \eta_K $ is fulfilled, with $ \sigma $ an a priori chosen threshold, e.g., $ \sigma = 0.5 $. Note that $ \sigma = 0 $ results in uniform refinement. In Figure~\ref{fig:example}, we observe indeed reduced convergence rates for all polynomial degrees and dimensions tested. However, using an adaptive procedure, we are able to recover the optimal rates. Moreover, we significantly reduce the number of dofs required to reach a certain error threshold. For instance, in the case $ d=3 $ and $ p=2 $, we need $ 276\,922\,881 $ dofs to get an error in the $ \|\,.\,\|_h $-norm of $ \sim10^{-1} $, whereas we only need $ 26\,359 $ dofs with adaptive refinement. In terms of runtime, the uniform refinement needed $ 478.57 s $ for assembling and solving, while the complete adaptive procedure took $ 156.5 s $ only. 
The parallel performance 
is also shown in Figure~\ref{fig:example}, where we obtain a nice strong scaling up to 64 cores. Then the local problems are too small (only $ \sim 10\,000 $ dofs for 128 cores) and the communication overhead becomes too large. 
\begin{figure}[ht]
    \centering%
    \begin{minipage}{0.5\textwidth}%
        \centering%
        \includegraphics[width=\textwidth]{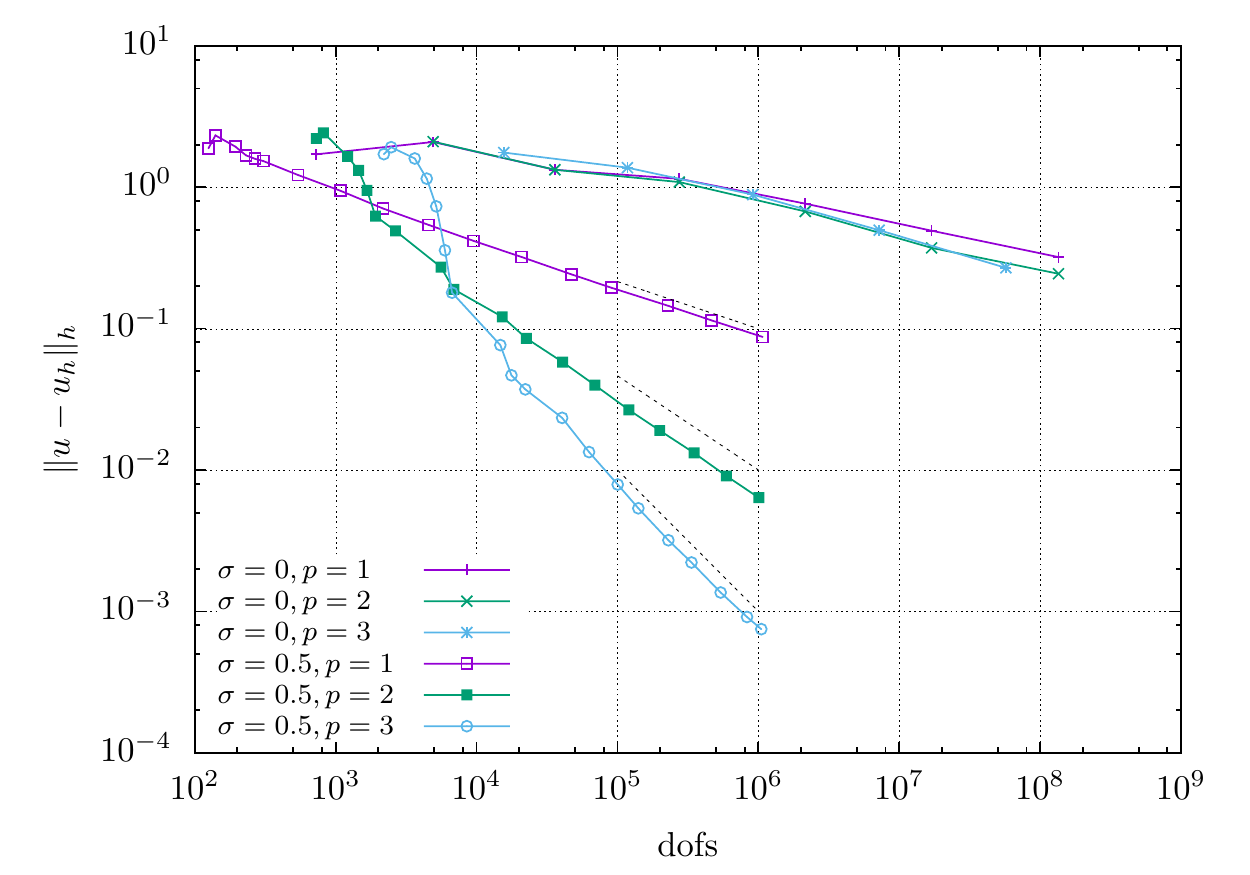}
    \end{minipage}%
    \hfill%
    \begin{minipage}{0.5\textwidth}%
        \centering%
        \includegraphics[width=\textwidth]{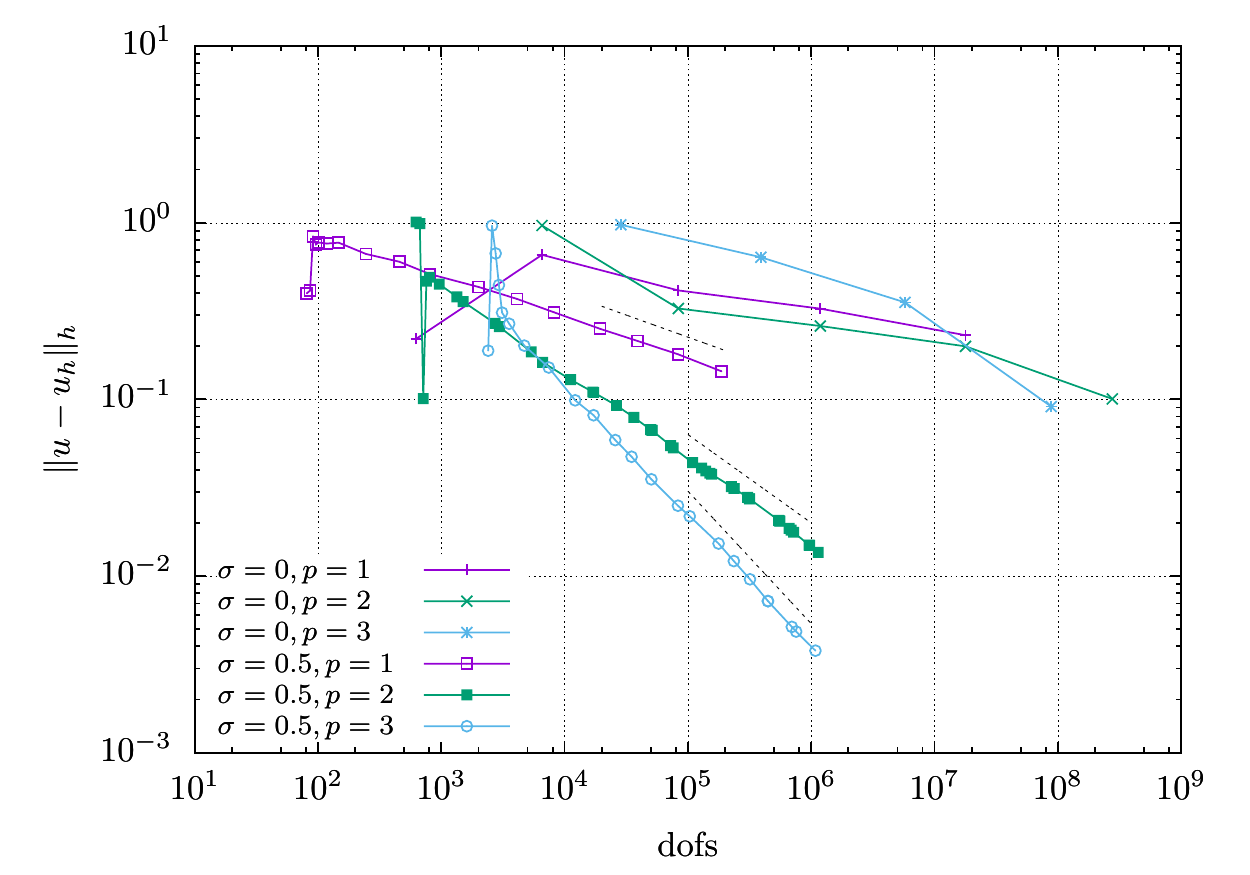}
    \end{minipage}%
    \vfill%
    \begin{minipage}{0.5\textwidth}%
        \centering%
        \includegraphics[width=.6\textwidth]{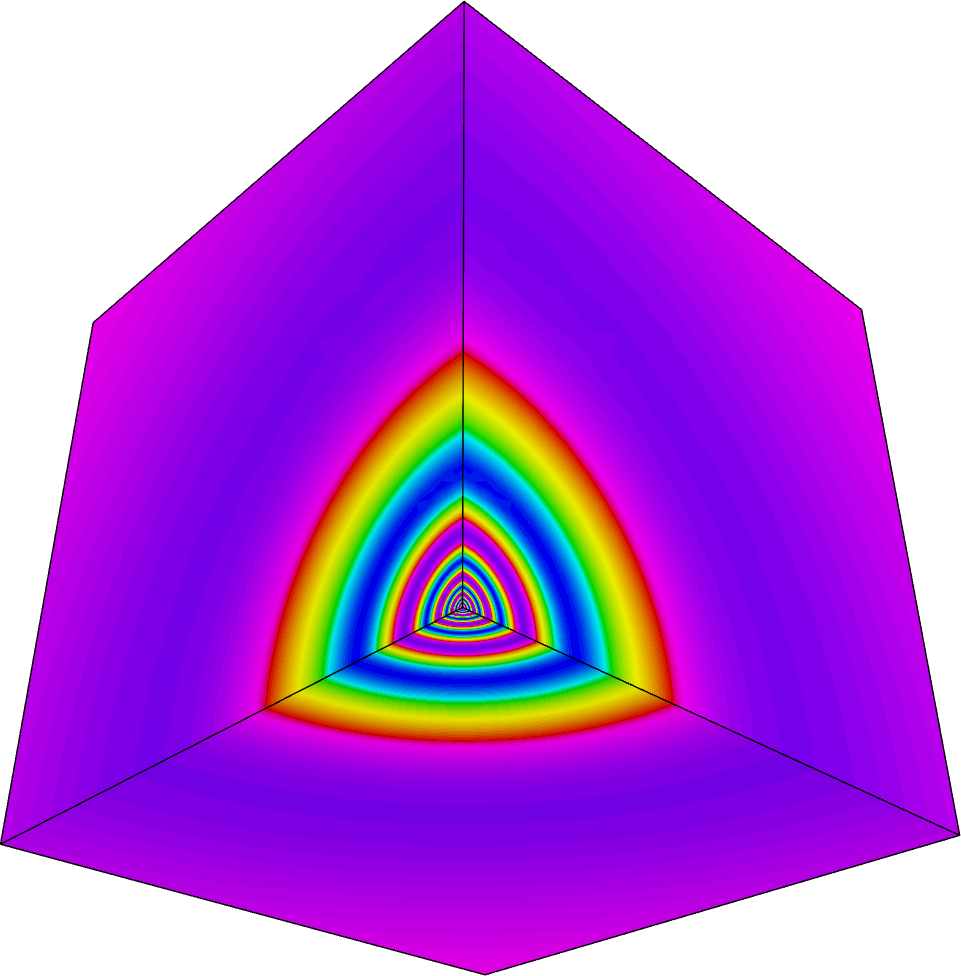}
    \end{minipage}%
    \hfill%
    \begin{minipage}{0.5\textwidth}%
        \centering%
        \includegraphics[width=\textwidth]{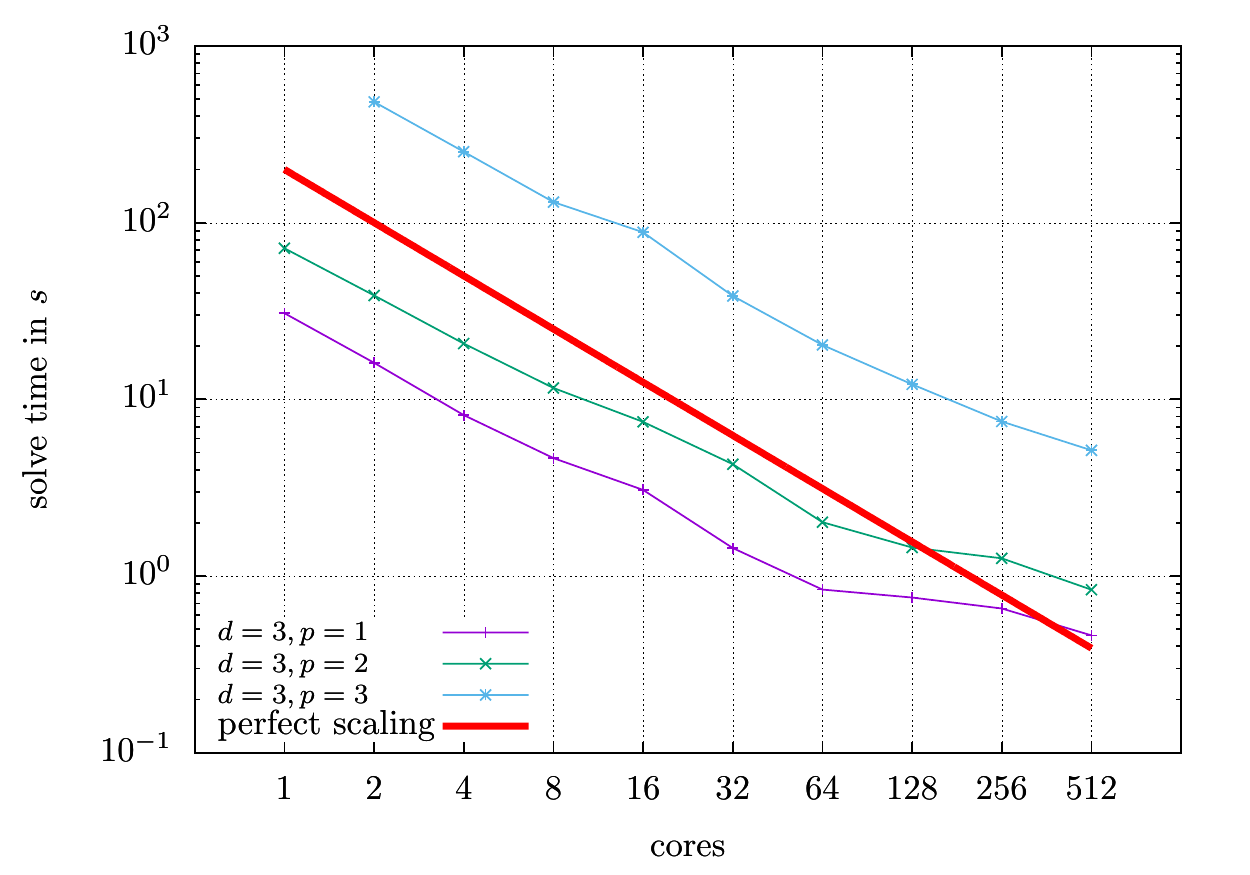}
    \end{minipage}%
    \caption{Example~\ref{sec:num:1}: Error rates in the $ \|\,.\,\|_h $-norm for $ d=2 $ (upper left) and $ d=3 $ (upper right), the dotted lines indicate the optimal rate; Plot of the approximate solution $ u_h $ centered at the origin $ (0,0,0) $ \cite{LS:LangerNeumuellerSchafelner:2019a} (lower left); Strong scaling on a mesh with $ 1\,185\,921 $ dofs for $ p=1,2 $ and $ 5\,764\,801 $ dofs for $ p=3 $ (lower right).}
    \label{fig:example}
\end{figure}
\subsection{Example: Moving peak}\label{sec:num:2}
For the second example, we consider the unit-cube $ Q = (0,1)^3 $, i.e. $ d=2 $. As diffusion coefficient, the choice $ \nu\equiv 1 $ is made. We choose the function
\[ u(x,t) = (x_1^2-x_1)(x_2^2-x_2)e^{-100((x_1-t)^2+(x_2-t)^2)}, \]
as exact solution and compute all data accordingly. This function is smooth, and almost zero everywhere, except in a small region around the line from the origin $ (0,0,0) $ to $ (1,1,1) $. This motivates again the use of an adaptive method. We use the residual based indicator $ \eta_K $ introduced in Example~\ref{sec:num:1}. In Figure~\ref{fig:example2}, we can observe that we indeed obtain optimal rates for both uniform and adaptive refinement. However, using the a posteriori error indicator, we reduce the number of dofs needed to reach a certain threshold by one order of magnitude. For instance, in the case $ p=2 $, we need $ 16\,974\,593 $ dofs to obtain an error of $ \sim 7\times10^{-5} $ with uniform refinement. Using adaptive refinement, we need $ 1\,066\,777 $ dofs only.
\begin{figure}[ht]
    \centering%
    \begin{minipage}{0.5\textwidth}%
        \centering%
        \includegraphics[width=\textwidth]{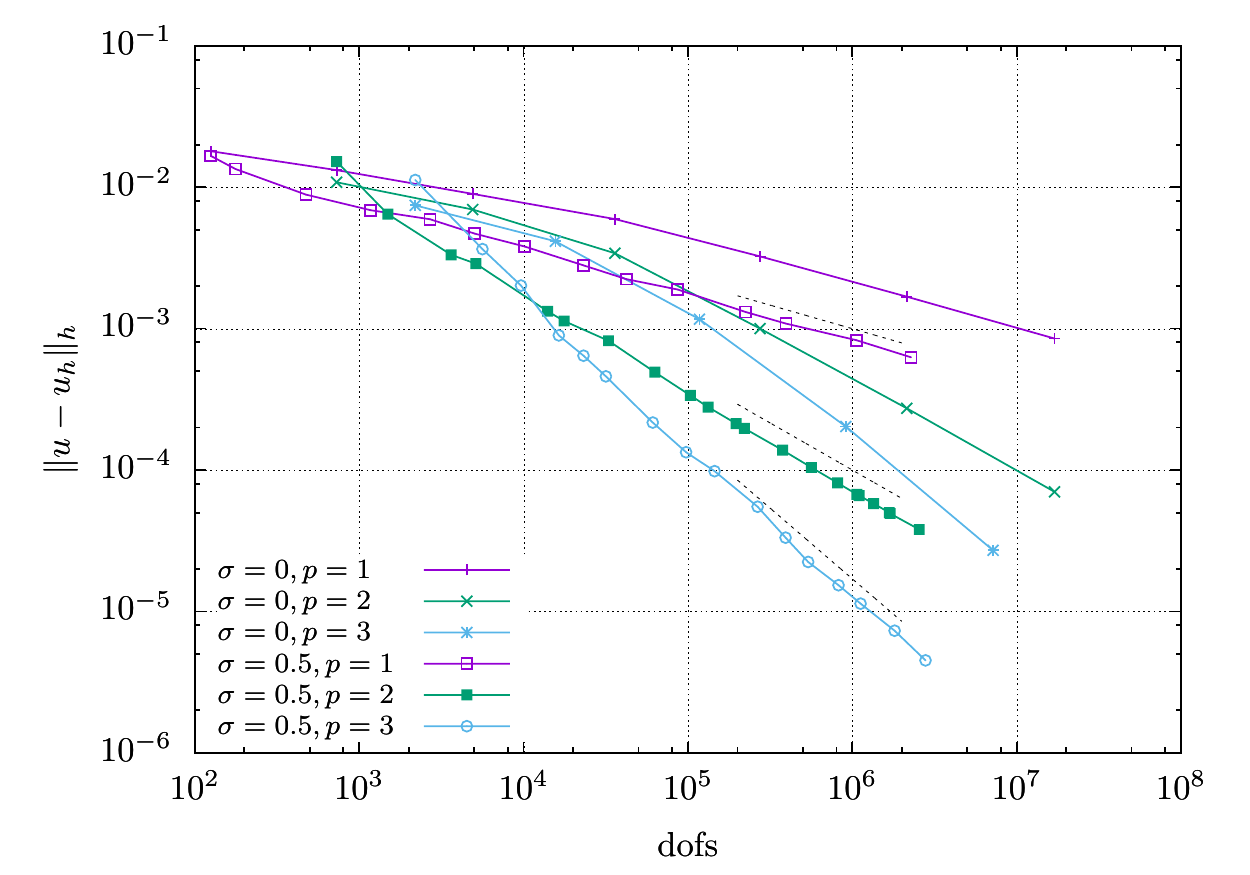}
    \end{minipage}%
    \hfill%
    \begin{minipage}{0.5\textwidth}%
        \centering%
        \includegraphics[width=.65\textwidth]{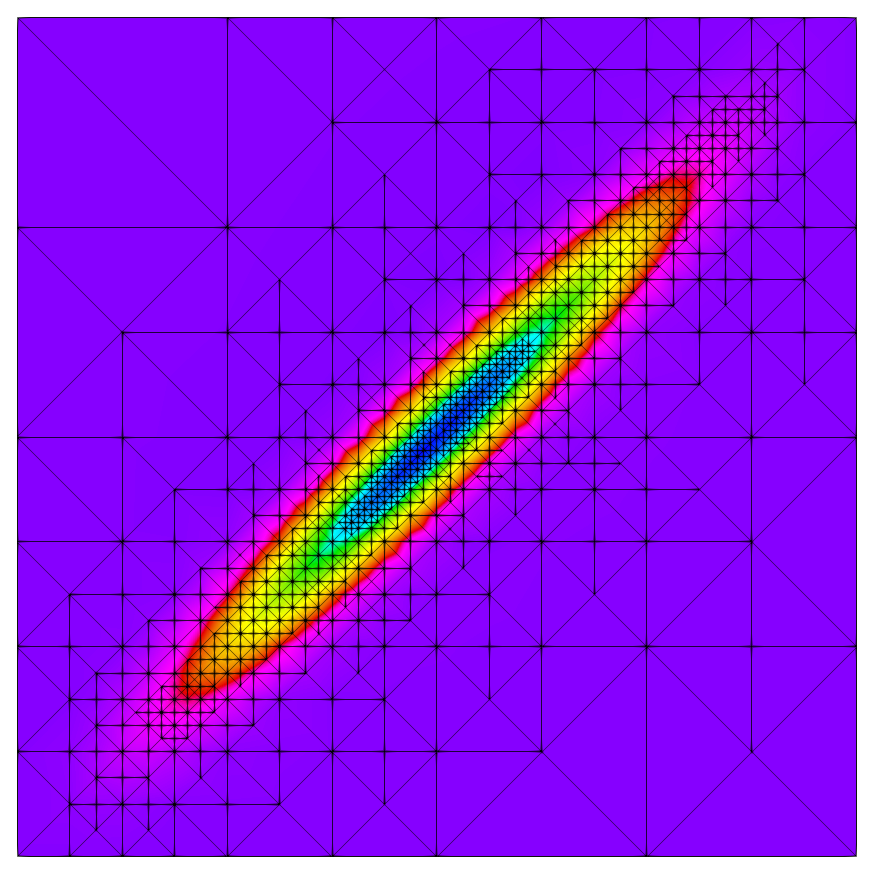}
    \end{minipage}%
    \caption{Example~\ref{sec:num:2}: Error rates in the $ \|\,.\,\|_h $-norm for $ d=2 $, the dotted lines indicate the optimal rates (left); Diagonal cut through the space-time mesh $ \mathcal{T}_h $ along the line from $ (0,0,0) $ to $ (1,1,1) $ after 8 adaptive refinements (right). }%
    \label{fig:example2}
\end{figure}
%
\section{Conclusions}\label{sec:conc}
%
Following \cite{LS:LangerNeumuellerSchafelner:2019a}, we introduced a space-time finite element solver for non-auto\-no\-mous parabolic evolution problems on completely unstructured simplicial meshes. We only assumed that we have so-called maximal parabolic regularity, i.e., the PDE is well posed in $ L_2 $. We note that this property is only required locally in order to derive a consistent space-time finite element scheme. We extended the a priori error estimate in the mesh-dependent energy norm to the case of non-smooth solutions, i.e. $ u\in H^{1+\epsilon}(Q) $, with $ 0 < \epsilon \le 1 $. This is necessary, since
we cannot expect a smooth solution, especially when dealing with discontinuous diffusion coefficients. For simplicity, we considered only piecewise constant diffusion coefficients. The extension to piecewise smooth coefficients is straight-forward, but more technical. In comparison to the previous result for sufficiently smooth solutions, we no longer have a completely localized estimate. We also have to include the neighborhood of an element $ K\in\mathcal{T}_h $. This may not be sufficient if we have to deal with solutions that have different regularity in different subdomains of the whole space-time cylinder. In applications, these subdomains correspond, for instance, to different materials. However, Duan et al. \cite{LS:DuanLiTanZheng:2012a} have shown a quasi-interpolation estimate that fits into this setting. \\
We performed two numerical examples with known solutions. The first example had a highly oscillatory solution, and the second one was almost zero everywhere except along a line through the space-time cylinder. Using a high-performance cluster, we solved both problems on a sequence of uniformly refined meshes, where we also obtained good strong scaling rates. In order to reduce the computational cost, we also applied an adaptive procedure, using a residual based error indicator. Indeed, using a simultaneously in space and time adaptive procedure reduced the computational as well as the memory cost by a large factor. Moreover, we could observe that, especially for $ d=3 $, the AMG preconditioned GMRES method solves the problem quite efficiently. \\
%
%
\bibliographystyle{siam}
\bibliography{references}
%
\end{document}